\documentclass{article}

\usepackage{venndiagram}
\usepackage{graphicx}
\usepackage{
	amsmath,
	amssymb,
	amsthm
}
\usepackage{hyperref}
\usepackage{cleveref}
\usepackage{tikz}
\usepackage{enumerate}
\usetikzlibrary{matrix}
\usepackage{lscape}
\usetikzlibrary{automata,positioning}
\usepackage{caption}
\usepackage{multicol}
\usepackage{subcaption}
\usepackage{multicol}
\usepackage{float}
\usepackage{chngpage}
\usetikzlibrary{decorations.pathmorphing}
\tikzset{snake it/.style={decorate, decoration=snake}}
\usepackage{tikz-qtree}
\usepackage{mathtools}
\usepackage{pgf}
\usepackage{pgfplots}
\usepgfplotslibrary{fillbetween}

\usepackage[all]{xy}

\newtheorem{thm}{Theorem}
\newtheorem{df}[thm]{Definition}
\newtheorem{lem}[thm]{Lemma}
\newtheorem{cor}[thm]{Corollary}

\newtheorem{rem}[thm]{Remark}

\newcommand{\abs}[1]{\lvert#1\rvert}

\newcommand{\clpha}{\overline{\alpha}}

\DeclareMathOperator{\LZSet}{LZSet}
\DeclareMathOperator{\LZJD}{LZJD}

\begin{document}
	\title{
	Interpolating between the Jaccard distance and an analogue of the normalized information distance\thanks{This work was partially supported by grants from the Simons Foundation (\#704836 to Bj\o rn Kjos-Hanssen) and Decision Research Corporation
				(University of Hawai\textquoteleft i Foundation Account \#129-4770-4).}}
	\author{Bj{\o}rn Kjos-Hanssen}
	\maketitle
	\begin{abstract}
Jim{\'e}nez, Becerra, and Gelbukh (2013) defined a family of ``symmetric Tversky ratio models'' $S_{\alpha,\beta}$, $0\le\alpha\le 1$, $\beta>0$. Each function $D_{\alpha,\beta}=1-S_{\alpha,\beta}$ is a semimetric on the powerset of a given finite set.

We show that $D_{\alpha,\beta}$ is a metric if and only if $0\le\alpha \le \frac12$ and $\beta\ge 1/(1-\alpha)$.
This result is formally verified in the Lean proof assistant.

The extreme points of this parametrized space of metrics are
$\mathcal V_1=D_{1/2,2}$, the Jaccard distance, and
$\mathcal V_{\infty}=D_{0,1}$, an analogue of the normalized information distance of M.~Li, Chen, X.~Li, Ma, and Vit\'anyi (2004).

As a second interpolation, in general we also show that $\mathcal V_p$ is a metric, $1\le p\le\infty$, where
\[
\Delta_p(A,B)=(\abs{B\setminus A}^p+\abs{A\setminus B}^p)^{1/p},
\]
\[
\mathcal V_p(A,B)=\frac{\Delta_p(A,B)}{\abs{A\cap B} + \Delta_p(A,B)}.
\]
	\end{abstract}
	\section{Introduction}

		Distance measures (metrics), are used in a wide variety of scientific contexts.
		In bioinformatics, M.~Li, Badger, Chen, Kwong, and Kearney \cite{Li2001AnIS} introduced an information-based sequence distance.
		In an information-theoretical setting, M.~Li, Chen, X.~Li, Ma and Vit{\'a}nyi \cite{MR2103495} rejected the distance of \cite{Li2001AnIS} in favor of a
		\emph{normalized information distance} (NID).
		The Encyclopedia of Distances \cite{MR3559482} describes the NID on page 205 out of 583, as
		\[
			\frac{
				\max\{
					K(x\mid y^*),K(y\mid x^*)
				\}
			}{
				\max\{
					K(x),K(y)
				\}
			}
		\]
		where $K(x\mid y^*)$ is the Kolmogorov complexity of $x$ given a shortest program $y^*$ to compute $y$.
		It is equivalent to be given $y$ itself in hard-coded form:
		\[
			\frac{
				\max\{
					K(x\mid y),K(y\mid x)
				\}
			}{
				\max\{
					K(x),K(y)
				\}
			}
		\]
		Another formulation (see \cite[page 8]{MR2103495}) is
		\[
			\frac{K(x,y)-\min\{K(x),K(y)\}}{\max\{K(x),K(y)\}}.
		\]

		The fact that the NID is in a sense a normalized metric is proved in~\cite{MR2103495}.
		Then in 2017, while studying malware detection, Raff and Nicholas \cite{Raff2017AnAT} suggested Lempel--Ziv Jaccard distance (LZJD) as a practical alternative to NID.
		As we shall see, this is a metric.
		In a way this constitutes a full circle: the distance in \cite{Li2001AnIS} is itself essentially a Jaccard distance,
		and the LZJD is related to it as Lempel--Ziv complexity is to Kolmogorov complexity.
		In the present paper we aim to shed light on this back-and-forth by showing that
		the NID and Jaccard distances constitute the endpoints of a parametrized family of metrics.

		For comparison, the Jaccard distance between two sets $X$ and $Y$, and our analogue of the NID, are as follows:
		\begin{eqnarray}
			J_1(X,Y)&=&\frac{\abs{X\setminus Y}+\abs{Y\setminus X}}{\abs{X\cup Y}} = 1 - \frac{\abs{X\cap Y}}{\abs{X\cup Y}}\\
			J_{\infty}(X,Y)&=&\frac{\max\{\abs{X\setminus Y}, \abs{Y\setminus X}\}}{\max\{\abs{X},\abs{Y}\}}\label{setNID}
		\end{eqnarray}
		Our main result \Cref{main} shows which interpolations between these two are metrics.
		The way we arrived at $J_{\infty}$ as an analogue of NID is via Lempel--Ziv complexity.
		While there are several variants \cite{MR389403,MR530215,MR507465}, the LZ 1978 complexity \cite{MR507465} of a sequence is the cardinality of a certain set, the dictionary.
		\begin{df}
			Let $\LZSet(A)$ be the Lempel--Ziv dictionary for a sequence $A$.
			We define
			LZ--Jaccard distance $\LZJD$ by
			\[
				\LZJD(A,B) = 1- \frac{\abs{\LZSet(A)\cap \LZSet(B)}}{\abs{\LZSet(A)\cup \LZSet(B)}}.
			\]
		\end{df}
		It is shown in \cite[Theorem 1]{Li2001AnIS} that the triangle inequality holds for a function which they call an information-based sequence distance.
		Later papers give it the notation $d_s$ in \cite[Definition V.1]{MR2103495}, and call their normalized information distance $d$.
		Raff and Nicholas \cite{Raff2017AnAT} introduced the LZJD and did not discuss the appearance of $d_s$ in \cite[Definition V.1]{MR2103495},
		even though they do cite \cite{MR2103495} (but not \cite{Li2001AnIS}).

		Kraskov et al.~\cite{Kraskov2003HierarchicalCB,Kraskov_2005} use $D$ and $D'$ for continuous analogues of $d_s$ and $d$ in \cite{MR2103495} (which they cite).
		The \emph{Encyclopedia} calls it the normalized information metric,
		\[
			\frac{H(X\mid Y)+H(X\mid Y)}{H(X,Y)} = 1 - \frac{I(X;Y)}{H(X,Y)}
		\]
		or Rajski distance \cite{MR0134805}.

		This $d_s$ was called $d$ by \cite{Li2001AnIS} --- see \Cref{favored}.
		\begin{table}
			\begin{tabular}{l l l}
				Reference					&	Jaccard notation 	& NID notation\\
				\cite{Li2001AnIS}			&	$d$		&\\
				\cite{MR2103495}			&	$d_s$	& $d$\\
				\cite{Kraskov_2005}			&	$D$		& $D'$\\
				\cite{Raff2017AnAT}			&	$\LZJD$	& NCD
			\end{tabular}
			\caption{Overview of notation used in the literature. (It seems that authors use simple names for their favored notions.)}\label{favored}
		\end{table}
		Conversely, \cite[near Definition V.1]{MR2103495} mentions mutual information.

		\begin{rem}
			Ridgway \cite{wiki:xxx} observed that the entropy-based distance $D$  is essentially a
			Jaccard distance. No explanation was given, but we attempt one as follows.
			Suppose $X_1,X_2,X_3,X_4$ are iid Bernoulli($p=1/2$) random variables, $Y$ is the random vector $(X_1,X_2,X_3)$ and $Z$ is $(X_2,X_3,X_4)$.
			Then $Y$ and $Z$ have two bits of mutual information $I(Y,Z)=2$.
			They have an entropy $H(Y)=H(Z)=3$ of three bits. Thus the relationship $H(Y,Z)=H(Y)+H(Z)-I(Y,Z)$
			becomes a Venn diagram relationship $\abs{\{X_1,X_2,X_3,X_4\}} = \abs{\{X_1,X_2,X_3\}} + \abs{\{X_2,X_3,X_4\}} - \abs{\{X_2,X_3\}}$.
			The relationship to Jaccard distance may not have been well known, as it is not mentioned in \cite{Kraskov_2005,10.1109/TKDE.2007.48,Li2001AnIS,1412045}.
		\end{rem}

		A more general setting is that of STRM (Symmetric Tversky Ratio Models), \Cref{symm}.
		These are variants of the Tversky index (\Cref{unsymm}) %with $\alpha=0$ and $\beta=1$
		proposed in \cite{DBLP:conf/starsem/JimenezBG13}.

		\subsection{Generalities about metrics}
		\begin{df}\label{october2020}
		Let $\mathcal X$ be a set.
			A \emph{metric} on $\mathcal X$ is a function
			$d : \mathcal X \times \mathcal X \to \mathbb R$ such that
			\begin{enumerate}
				\item\label{met-nonneg} $d(x, y) \ge 0$,
				\item\label{met-0} $d(x, y) = 0$ if and only if $x = y$,
				\item\label{met-symm} $d(x, y) = d(y, x)$ (symmetry),
				\item\label{met-tri} $d(x,y)\le d(x,z)+d(z,y)$ (the triangle inequality)
			\end{enumerate}
			for all $x,y,z\in\mathcal X$.
			If $d$ satisfies \Cref{met-nonneg}, \Cref{met-0}, \Cref{met-symm} but not necessarily \Cref{met-tri} then $d$ is called a \emph{semimetric}.
		\end{df}
		A basic exercise in \Cref{october2020} that we will make use of is \Cref{lin-comb-metric}.
		\begin{thm}\label{lin-comb-metric}
			 If $d_1$ and $d_2$ are metrics and $a,b$ are nonnegative constants, not both zero, then $ad_1+bd_2$ is a metric.
		\end{thm}
		\begin{proof}
			\Cref{met-nonneg} is immediate from \Cref{met-nonneg} for $d_1$ and $d_2$.

			\Cref{met-0}: Assume $ad_1(x,y)+bd_2(x,y)=0$. Then $ad_1(x,y)=0$ and $bd_2(x,y)=0$.
			Since $a,b$ are not both 0, we may assume $a>0$. Then $d_1(x,y)=0$ and hence $x=y$ by \Cref{met-0} for $d_1$.

			\Cref{met-symm}: We have $ad_1(x,y)+bd_2(x,y)=ad_1(y,x)+bd_2(y,x)$ by \Cref{met-symm} for $d_1$ and $d_2$.

			\Cref{met-tri}: By \Cref{met-tri} for $d_1$ and $d_2$ we have
			\begin{eqnarray*}
			ad_1(x,y)+bd_2(x,y)&\le& a(d_1(x,z)+d_1(z,y))+b(d_2(x,z)+d_2(z,y))\\ &=& (ad_1(x,z)+bd_2(x,z)) + (ad_2(z,y)+bd_2(z,y)).\qedhere			 
			\end{eqnarray*}
		\end{proof}
		\begin{lem}\label{mar29-2022}
			Let $d(x,y)$ be a metric and let $a(x,y)$ be a nonnegative symmetric function. If $a(x,z)\le a(x,y)+d(y,z)$ for all $x,y,z$,
			then $d'(x,y)=\frac{d(x,y)}{a(x,y)+d(x,y)}$,
			with $d'(x,y)=0$ if $d(x,y)=0$, is a metric.
		\end{lem}
		\begin{proof}
			As a piece of notation, let us write $d_{xy}=d(x,y)$ and $a_{xy}=a(x,y)$.
			As observated by \cite{210750}, in order to show
			\[
				\frac{d_{xy}}{a_{xy}+d_{xy}}\le  \frac{d_{xz}}{a_{xz}+d_{xz}} + \frac{d_{yz}}{a_{yz}+d_{yz}},
			\]
			it suffices to show the following pair of inequalities:
			\begin{eqnarray}
				\frac{d_{xy}}{a_{xy}+d_{xy}}\le &\frac{d_{xz}+d_{yz}}{a_{xy}+d_{xz}+d_{yz}}& \label{1}\\
				&\frac{d_{xz}+d_{yz}}{a_{xy}+d_{xz}+d_{yz}}& \le \frac{d_{xz}}{a_{xz}+d_{xz}} + \frac{d_{yz}}{a_{yz}+d_{yz}}\label{2}
			\end{eqnarray}
			Here \eqref{1} follows from $d$ being a metric, i.e., $d_{xy}\le d_{xz}+d_{yz}$, since
			\[
				c\ge 0<a\le b\implies \frac{a}{a+c}\le \frac{b}{b+c}.
			\]

			Next, \eqref{2} would follow from $a_{xy}+d_{yz}\ge a_{xz}$ and $a_{xy}+d_{xz}\ge a_{yz}$.
			By symmetry between $x$ and $y$ and since $a_{xy}=a_{yx}$ by assumption,
			it suffices to prove the first of these, $a_{xy}+d_{yz}\ge a_{xz}$, which holds by assumption.
		\end{proof}
        \subsection{Metrics on a family of finite sets}
        \begin{lem}\label{union-bound}
            For sets $A,B,C$, we have
            $\abs{A\setminus B}\le \abs{A\setminus C}+\abs{C\setminus B}$.
        \end{lem}
        \begin{proof}
            We have $A\setminus B\subseteq (A\setminus C)\cup (C\setminus B)$. Therefore, the result follows from the union bound for cardinality.
        \end{proof}
        \begin{lem}\label{jaccard-numerator}
             Let $f(A,B)=\abs{A\setminus B}+\abs{B\setminus A}$. Then $f$ is a metric.
        \end{lem}
        \begin{proof}
            The most nontrivial part is to prove the triangle inequality,
			\[
			\abs{A\setminus B}+\abs{B\setminus A}\le
			\abs{A\setminus C}+\abs{C\setminus A}
			+
			\abs{C\setminus B}+\abs{B\setminus C}.
			\]
			By the ``rotation identity'' $\abs{A\setminus C}+\abs{C\setminus B}+\abs{B\setminus A}=
			\abs{A\setminus B}+\abs{B\setminus C}+\abs{C\setminus A}$, this is equivalent to
			\[
			2(\abs{A\setminus B}+\abs{B\setminus A})\le
			2(\abs{A\setminus C}
			+
			\abs{C\setminus B}
			+
			\abs{B\setminus A}),
			\]
			which is immediate from \Cref{union-bound}.
        \end{proof}

        \begin{lem}\label{nid-numerator}
             Let $f(A,B)=\max\{\abs{A\setminus B},\abs{B\setminus A}\}$. Then $f$ is a metric.
        \end{lem}
        \begin{proof}
        For the triangle inequality, we need to show
			\[
			\max\{\abs{A\setminus B},\abs{B\setminus A}\}
			\le
			\max\{\abs{A\setminus C},\abs{C\setminus A}\}
			+
			\max\{\abs{C\setminus B},\abs{B\setminus C}\}.
			\]
			By symmetry we may assume that $\max\{\abs{A\setminus B},\abs{B\setminus A}\}=\abs{A\setminus B}$. Then, the result is immediate from \Cref{union-bound}.            
        \end{proof}
		For a real number $\alpha$, we write $\clpha=1-\alpha$.
		For finite sets $X,Y$ we define
		\[
			\tilde m(X,Y)=\min\{\abs{X\setminus Y}, \abs{Y\setminus X}\},
		\]
		\[
			\tilde M(X,Y)=\max\{\abs{X\setminus Y}, \abs{Y\setminus X}\}.
		\]
		\begin{lem}\label{empathy}
			Let $\delta:=\alpha \tilde m +\clpha \tilde M$. Let $X=\{0\}, Y=\{1\}, Z=\{0,1\}$. Then
			$\delta(X,Y)=1$, $\delta(X,Z)=\delta(Y,Z)=\clpha$.
		\end{lem}
		The proof of \Cref{empathy} is an immediate calculation.
		\begin{thm}\label{may17}
			$\delta_{\alpha}=\alpha \tilde m +\clpha \tilde M$ satisfies the triangle inequality  if and only if $0\le\alpha\le 1/2$.
		\end{thm}
		\begin{proof}
			We first show the \emph{only if} direction. By \Cref{empathy} the triangle inequality only holds for the example given there if $1\le 2\clpha$, i.e., $\alpha\le 1/2$.

			Now let us show the \emph{if} direction.
			%{\color{red}
			If $\alpha\le 1/2$ then $\alpha\le\clpha$, so $\delta_{\alpha}=\alpha (\tilde m + \tilde M) + (\clpha-\alpha) \tilde M$ is a nontrivial nonnegative linear combination.
			Since $(\tilde m+\tilde M)(A,B)=\abs{A\setminus B}+\abs{B\setminus A}$ (\Cref{jaccard-numerator}) and
			$\tilde M(A,B)=\max\{\abs{A\setminus B},\abs{B\setminus A}\}$ (\Cref{nid-numerator}) are both metrics, the result follows from \Cref{lin-comb-metric}.
		\end{proof}
		\begin{lem}\label{to-motivate}
			Suppose $d$ is a metric on a collection of nonempty sets $\mathcal X$, with $d(X,Y)\le 2$ for all $X,Y\in\mathcal X$.
			Let $\hat{\mathcal X}=\mathcal X\cup\{\emptyset\}$ and define $\hat d:\hat{\mathcal X}\times\hat{\mathcal X}\to\mathbb R$ by
			stipulating that for $X,Y\in\mathcal X$,
			\[
				\hat d(X,Y)=d(X,Y);\quad
				d(X,\emptyset)=1=d(\emptyset,X);\quad
				d(\emptyset,\emptyset)=0.
			\]
			Then $\hat d$ is a metric on $\hat{\mathcal X}$.
		\end{lem}
		\begin{thm}\label{mar29-22-11am}
			Let $f(A,B)$ be a metric such that
			\[
				\abs{B\setminus A}\le f(A,B)
			\] for all $A,B$. Then the function $d$ given by
			\[
				d(A,B)=\begin{cases}
				\frac{f(A,B)}{\abs{A\cap B}+f(A,B)},&\text{if }\abs{A\cap B}+f(A,B)>0,\\
				0,&\text{otherwise},
				\end{cases}
			\]
			is a metric.
		\end{thm}
		\begin{proof}
			By \Cref{mar29-2022} (with $a_{x,y}=\abs{X\cap Y}$) we only need to verify that for all sets $A,B,C$,
			\[
			\abs{A\cap C}+f(A,B)\ge \abs{B\cap C}.
			\]
			And indeed, since tautologically $B\cap C\subseteq (B\setminus A)\cup (A\cap C)$, by the union bound we have
		$\abs{B\cap C}-\abs{A\cap C}\le \abs{B\setminus A}\le f(A,B)$.
		\end{proof}
		\begin{thm}\label{too-similar}
			Let $f(A,B)=m\min\{\abs{A\setminus B},\abs{B\setminus A}\}+M\max\{\abs{A\setminus B},\abs{B\setminus A}\}$ with $0<m\le M$ and $1\le M$.
			Then the function $d$ given by
			\[
			d(A,B)=
			\begin{cases}
				\frac{f(A,B)}{\abs{A\cap B}+ f(A,B)},&\text{if } A\cup B\ne\emptyset,\\
				0,&\text{otherwise},
			\end{cases}
			\]
			is a metric.
		\end{thm}
		\begin{proof}
			We have
			$f(A,B)=(m+M) \delta_{\alpha}(A,B)$ where $\alpha=\frac{m}{m+M}$. Since $m\le M$, $\alpha\le 1/2$, so $f$ satisfies the triangle inequality by \Cref{may17}.
			Since $m>0$, in fact $f$ is a metric.
			Using $M\ge 1$,
			\[f(A,B)\ge M\max\{\abs{A\setminus B},\abs{B\setminus A}\}\ge M \abs{B\setminus A}\ge \abs{B\setminus A},
			\]
			so that by \Cref{mar29-22-11am}, $d$ is a metric.
		\end{proof}
	\subsection{Tversky indices}

		\begin{df}[{\cite{tversky}}]\label{unsymm}
			For sets $X$ and $Y$ the Tversky index with parameters $\alpha,\beta\ge 0$ is a number between 0 and 1 given by
			\[
				S(X, Y) = \frac{\abs{X \cap Y}}{\abs{X \cap Y} + \alpha\abs{X \setminus Y} + \beta\abs{Y \setminus X}}.
			\]
			We also define the corresponding Tversky dissimilarity $d^T_{\alpha,\beta}$ by
			\[
				d^{T}_{\alpha,\beta}(X,Y) = 
				\begin{cases}
					1-S(X,Y)&\text{if }X\cup Y\ne\emptyset;\\
					0&\text{if }X=Y=\emptyset.
				\end{cases}
			\]
		\end{df}

		\begin{df}
			The Szymkiewicz–-Simpson coefficient is defined by
			\[
				 \operatorname{overlap}(X,Y) = \frac{\abs{X \cap Y}}{\min(\abs{X},\abs{Y})}
			\]
		\end{df}
		We may note that $\operatorname{overlap}(X,Y)=1$ whenever $X\subseteq Y$ or $Y\subseteq X$, so that $1-\operatorname{overlap}$ is not a metric.
		\begin{df}\label{dice}
			The S{\o}rensen--Dice coefficient is defined by
			\[
				\frac{2\abs{X\cap Y}}{\abs{X}+\abs{Y}}.
			\]
		\end{df}
		\begin{df}[{\cite[Section 2]{DBLP:conf/starsem/JimenezBG13}}]\label{symm}
			Let $\mathcal X$ be a collection of finite sets. We define $S:\mathcal X\times\mathcal X\to \mathbb R$ as follows.
			The symmetric Tversky ratio model is defined by
			\[
				\mathbf{strm}(X,Y)=\frac{| X \cap Y |+\mathrm{bias}}{| X \cap Y |+\mathrm{bias}+\beta\left(\alpha \tilde m+(1-\alpha)\tilde M\right)}
			\]
			The unbiased symmetric TRM ($\mathbf{ustrm}$) is the case where $\mathrm{bias}=0$, which is the case we shall assume we are in for the rest of this paper.
			The Tversky semimetric $D_{\alpha,\beta}$ is defined by $D_{\alpha,\beta}(X,Y)=1-\mathbf{ustrm}(X,Y)$, or more precisely
			\[
			D_{\alpha,\beta}(X,Y) = 
			\begin{cases}
				\beta \frac{\alpha \tilde m + (1-\alpha)\tilde M}{\abs{X\cap Y}+\beta(\alpha \tilde m + (1-\alpha)\tilde M)}, & \text{if }X\cup Y\ne\emptyset;\\
				0 & \text{if }X=Y=\emptyset.
			\end{cases}
			\]
		\end{df}
		Note that for $\alpha = 1/2$, $\beta = 1$,
		the STRM is equivalent to the S{\o}rensen--Dice coefficient.
		Similarly,
		for $\alpha = 1/2$, $\beta = 2$, it is equivalent to Jaccard's coefficient.

		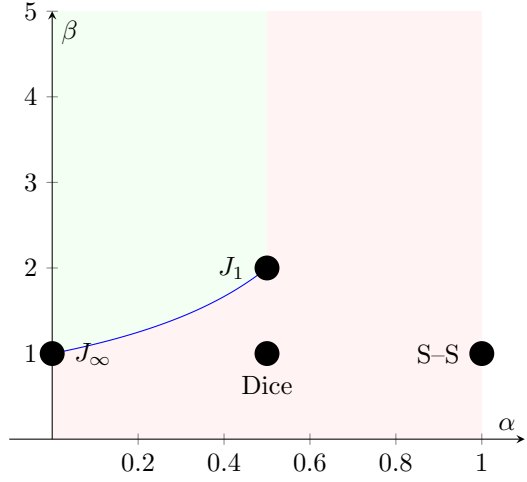
\begin{figure}
			\centering
			\begin{tikzpicture}[>=stealth]
				\begin{axis}[
					xmin=-0.1,xmax=1.1,
					ymin=0,ymax=5,
					axis x line=middle,
					axis y line=middle,
					axis line style=->,
					xlabel={$\alpha$},
					ylabel={$\beta$},
				]
					\addplot[name path=h,no marks,blue,<->] expression[domain=0:0.5,samples=100]{1/(1-x)};
					\path[name path=myaxis] (axis cs:0,5) -- (axis cs:0.5,5);
					\path[name path=topred] (axis cs:0.5,5) -- (axis cs:1,5);
					\path[name path=botred] (axis cs:0.5,0) -- (axis cs:1,0);
					\path[name path=botgreen](axis cs:0,0) -- (axis cs:0.5,0);

					\addplot [
							thick,
							color=green,
							fill=green, 
							fill opacity=0.05
						]
						fill between[
							of=h and myaxis,
							%soft clip={domain=0:1},
						];
					\addplot [
							thick,
							color=red,
							fill=red, 
							fill opacity=0.05
						]
						fill between[
							of=topred and botred,
							%soft clip={domain=0:1},
						];
					\addplot [
							thick,
							color=red,
							fill=red, 
							fill opacity=0.05
						]
						fill between[
							of=h and botgreen,
							%soft clip={domain=0:1},
						];

					\node[label={0:{$J_\infty$ }},circle,fill] at (axis cs:0,1) {};
					\node[label={180:{ $J_1$ }},circle,fill] at (axis cs:0.5,2) {};
					\node[label={270:{Dice}},circle,fill] at (axis cs:0.5,1) {};
					\node[label={180:{S--S}},circle,fill] at (axis cs:1,1) {};

				\end{axis}
			\end{tikzpicture}
			\caption{A Tversky semimetric $D_{\alpha,\beta}$ is a metric if and only if $(\alpha,\beta)$ belongs to the green region.
			The parameter values corresponding to
			the Jaccard distance $J_1$, the analogue of normalized information distance analogue $J_{\infty}$,
			the S{\o}rensen--Dice semimetric, and
			the Szymkiewicz–-Simpson semimetric are indicated.}\label{niceGraph}
		\end{figure}

	\section{Tversky metrics}
		\begin{thm}\label{well}
			The function $D_{\alpha,\beta}$ is a metric only if $\beta\ge 1/(1-\alpha)$.
		\end{thm}
		\begin{proof}
			Recall that with $D=D_{\alpha,\beta}$,
			\[
				D(X,Y)=\frac{\beta\delta}{\abs{X\cap Y}+\beta\delta}.
			\]
			By \Cref{empathy}, for the example given there we have
			\begin{eqnarray*}
				D(X,Y) &=& \frac{\beta\cdot 1}{0+\beta\cdot 1} = 1,
			\\
				D(X,Z) = D(Y,Z) &=& \frac{\beta\cdot\clpha}{1+\beta\cdot \clpha}.
			\end{eqnarray*}
			The triangle inequality is then equivalent to:
			\[
				1\le 2\frac{\beta\clpha}{1+\beta\clpha}\iff \beta\clpha\ge 1
		\iff\beta\ge 1/(1-\alpha).\qedhere
			\]
		\end{proof}
		In \Cref{pyEx} we use the interval notation on $\mathbb N$, given by $[a,a]=\{a\}$ and $[a,b]=[a,b-1]\cup\{b\}$.
		\begin{thm}\label{pyEx}
			The function $D_{\alpha,\beta}$ is a metric on all finite power sets
			only if $\alpha\le 1/2$.
		\end{thm}
		\begin{proof}
			Suppose $\alpha>1/2$.
			Then $2\clpha<1$.
			Let $n$ be an integer with $n>\frac{\beta\clpha}{1-2\clpha}$.		Let $X_n=[0,n]$, and $Y_n=[1,n+1]$, and $Z_n=[1,n]$.
			The triangle inequality says
			\begin{eqnarray*}
				\beta\frac{1}{n+\beta\cdot 1} = D(X_n,Y_n)&\le& D(X_n,Z_n)+D(Z_n,Y_n)=2\beta \frac{\clpha}{n+\beta\clpha}
			\\
				n+\beta\clpha &\le& 2 \clpha(n+\beta)
			\\
				n(1-2\clpha)&\le& \beta\clpha
			\end{eqnarray*}

			Then the triangle inequality does not hold, so $D_{\alpha,\beta}$ is not a metric on
			the power set of $[0,n+1]$.
		\end{proof}

		\begin{thm}\label{main}
			Let $0\le\alpha\le 1$ and $\beta> 0$.
			Then $D_{\alpha,\beta}$ is a metric if and only if $0\le\alpha\le 1/2$ and $\beta\ge 1/(1-\alpha)$.
		\end{thm}
		\begin{proof}
		\Cref{well} and \Cref{pyEx} give the necessary condition. Since
			\[
			D_{\alpha,\beta} = 
			\begin{cases}
				\beta \frac{\alpha \tilde m + (1-\alpha)\tilde M}{\abs{X\cap Y}+\beta(\alpha \tilde m + (1-\alpha)\tilde M)}, & \text{if }X\cup Y\ne\emptyset,\\
				0 & \text{otherwise},
			\end{cases}
			\]
			where $\tilde m$ is the minimum of the set differences and $\tilde M$ is the maximum,
			we can let $f(A,B)=\beta(\alpha \tilde m(A,B)+(1-\alpha)\tilde M(A,B))$.
			Then with the constants $m=\beta\alpha$ and $M=\beta\clpha$, we can apply \Cref{too-similar}.
		\end{proof}

		\Cref{main} is illustrated in \Cref{niceGraph}. We have formally proved \Cref{main} in the Lean theorem prover.
		The Github repository can be found at \cite{git}.

	\subsection{A converse to Gragera and Suppakitpaisarn}
		\begin{thm}[{Gragera and Suppakit\-paisarn \cite{MR3492464,MR3775059}}]\label{nother}
			 The optimal constant $\rho$ such that $d^T_{\alpha,\beta}(X,Y)\le \rho(d^T_{\alpha,\beta}(X,Y)+d^T_{\alpha,\beta}(Y,Z))$ for all $X,Y,Z$
			 is
			\[
				\frac12\left(1+\sqrt{\frac{1}{\alpha\beta}}\right).
			\]
		\end{thm}
		\begin{cor}\label{GScor}
			$d^T_{\alpha,\beta}$ is a metric only if $\alpha=\beta\ge 1$.
		\end{cor}
		\begin{proof}
			Clearly, $\alpha=\beta$ is necessary to ensure $d^T_{\alpha,\beta}(X,Y)=d^T_{\alpha,\beta}(Y,X)$.
			Moreover $\rho\le 1$ is necessary, so \Cref{nother} gives $\alpha\beta\ge 1$.
		\end{proof}
		\Cref{cool} gives the converse to the Gragera and Suppakit\-paisarn inspired \Cref{GScor}:
		\begin{thm}\label{cool}
			The Tversky dissimilarity $d^T_{\alpha,\beta}$ is a metric iff $\alpha=\beta\ge 1$.
		\end{thm}
		\begin{proof}
			Suppose the Tversky dissimilarity $d^T_{\alpha,\beta}$ is a semimetric.
			Let $X,Y$ be sets with $\abs{X\cap Y}=\abs{X\setminus Y}=1$ and $\abs{Y\setminus X}=0$.
			Then
			\[
				1-\frac1{1+\beta} = d^T_{\alpha,\beta}(Y,X) = d^T_{\alpha,\beta}(X,Y)=1-\frac{1}{1+\alpha},
			\]
			hence $\alpha=\beta$. Let $\gamma=\alpha=\beta$.

			Now, $d^T_{\gamma,\gamma}=D_{\alpha_0,\beta_0}$ where $\alpha_0=1/2$ and $\beta_0 = 2\gamma$. Indeed, with 
			$\tilde m=\min\{\abs{X\setminus Y}, \abs{Y\setminus X}\}$ and $\tilde M=\max\{\abs{X\setminus Y}, \abs{Y\setminus X}\}$, since
			\[
				D_{\alpha_0,\beta_0} = \beta_0\frac{
					\alpha_0 \tilde m + (1-\alpha_0)\tilde M
				}{
					\abs{X\cap Y} + \beta_0\left[\alpha_0 \tilde m + (1-\alpha_0)\tilde M\right]
				},
			\]
			\[
				D_{\frac12,2\gamma} = 2\gamma\frac{
					\frac12 \tilde m + (1-\frac12)\tilde M
				}{
					\abs{X\cap Y} + 2\gamma\left[\frac12 \tilde m + (1-\frac12)\tilde M\right]
				}
			\]
			\[
			= \gamma\frac{
					\abs{X\setminus Y}+\abs{Y\setminus X}
							}{
								\abs{X\cap Y} + \gamma\left[\abs{X\setminus Y}+\abs{Y\setminus X}\right]
							}
				= 1 - \frac{\abs{X \cap Y}}{\abs{X \cap Y} + \gamma\abs{X \setminus Y} + \gamma\abs{Y \setminus X}} = d^T_{\gamma,\gamma}.
			\]
			By \Cref{main}, $d^T_{\gamma,\gamma}$ is a metric if and only if $\beta_0\ge 1/(1-\alpha_0)$.
			This is equivalent to $2\gamma \ge 2$, i.e., $\gamma\ge 1$.
		\end{proof}
		The truth or falsity of \Cref{cool} does not arise in Gragera and Suppakit\-paisarn's work,
		as they require $\alpha,\beta\le 1$ in their definition of Tversky index.
		We note that Tversky \cite{tversky} only required $\alpha,\beta\ge 0$.

	\section{Lebesgue-style metrics}
		Incidentally, the names of $J_1$ and $J_{\infty}$ come from the observation that they are special cases of $J_p$ given by
		\[
			J_p(A,B)=
			\left(2\cdot
			\frac{
			\abs{B\setminus A}^p+\abs{A\setminus B}^p
			}{
			\abs{A}^p+\abs{B}^p+\abs{B\setminus A}^p+\abs{A\setminus B}^p
			}
			\right)^{1/p}
			=
			\begin{cases}
			J_1(A,B) & p=1\\
			J_{\infty}(A,B) & p\to\infty
			\end{cases}
		\]
		which was suggested in \cite{10.1007/978-3-030-93100-1_8} as another possible means of interpolating between $J_1$ and $J_{\infty}$.
		We still conjecture that $J_2$ is a metric, but shall not attempt to prove it here. However:
		\begin{thm}\label{thm:j3}
			 $J_3$ is not a metric.
		\end{thm}
		\begin{proof}
			Let $A,B,C$ be sets with
			$A\cap B\subseteq C\subseteq A\cup B$,
			\begin{eqnarray*}
				\abs{A\cap B}&=&80,\\
				\abs{A\setminus (B\cup C)}=\abs{B\setminus (A\cup C)}&=&10,\text{ and}\\
				\abs{A\cap C\setminus B}=\abs{B\cap C\setminus A}&=&11.
			\end{eqnarray*}
			\begin{center}
				\begin{tikzpicture}
					\draw (0,0) circle (1) node[above,shift={(0,1)}] {${A}$};
					\draw (1.2,0) circle (1) node[above,shift={(0,1)}] {${B}$};
					\draw (.6,-1.04) circle (1) node[shift={(1.1,-.6)}] {${C}$};
					\node at (0,-.7) {11};
					\node at (1.2,-.7) {11};
					\node at (1.4,.2) {10};
					\node at (-0.4,.2) {10};
					\node at (.6,.3) {0};
					\node at (.3,-1.5) {0};
					\node at (.6,-.35) {80};
				\end{tikzpicture}
			\end{center}
			Then since
			\[
			 \frac{21^3}{101^3 + 21^3}
			=0.0089086 \not\le 0.0089061
			=
			2^3 \left( \frac{10^3+11^3}{101^3+102^3+10^3+11^3}\right),
			\]
			we have
			\begin{eqnarray*}
			J_3(A,B)&=&\left(2 \frac{2\cdot 21^3}{2\cdot 101^3 + 2\cdot 21^3}\right)^{1/3}\\
			&\not\le&
			2 \left(2 \frac{10^3+11^3}{101^3+102^3+10^3+11^3}\right)^{1/3}
			= J_3(A,C)+J_3(C,B).\qedhere
			\end{eqnarray*}
		\end{proof}

		Because of \Cref{thm:j3}, we searched for a better version of $J_p$, and found $\mathcal V_p$:
		\begin{df}
		For each $1\le p\le\infty$, let\footnote{Here, $\mathcal V$ can stand for Paul M.~B.~Vit\'anyi, who introduced the %first 
		author to the normalized information distance at a Dagstuhl workshop in 2006.}
		\begin{eqnarray*}
		\Delta_p(A,B)&=&(\abs{B\setminus A}^p+\abs{A\setminus B}^p)^{1/p},\text{ and}
		\\
		\mathcal V_p(A,B)&=&\frac{\Delta_p(A,B)}{\abs{A\cap B} + \Delta_p(A,B)}.
		\end{eqnarray*}
		\end{df}
		We have $\mathcal V_1=J_1$ and $\mathcal V_{\infty}:=\lim_{p\to\infty}\mathcal V_p=J_{\infty}$.

		In a way what is going on here is that we consider $L^p$ spaces instead of
		\[
			\frac1p L^1 + \left(1-\frac1p\right)L^{\infty}
		\]
		spaces.

		\begin{thm}\label{fun-factory}
			 For each $1\le p\le\infty$, $\Delta_p$ is a metric.
		\end{thm}
		\begin{proof}
			The nontrivial part is to prove the triangle inequality
			\[
			\Delta_p(A,B)\le\Delta_p(A,C)+\Delta_p(C,B).
			\]
			We introduce seven variables for the cardinalities in the Venn diagram of $A,B,C$, as follows.
			\begin{center}
				\begin{tikzpicture}
					\draw (0,0) circle (1) node[above,shift={(0,1)}] {${A}$};
					\draw (1.2,0) circle (1) node[above,shift={(0,1)}] {${B}$};
					\draw (.6,-1.04) circle (1) node[shift={(1.1,-.6)}] {${C}$};
					\node at (0,-.7) {$w$};
					\node at (1.2,-.7) {$y$};
					\node at (1.4,.2) {$z$};
					\node at (-0.4,.2) {$v$};
					\node at (.6,.3) {$a$};
					\node at (.6,-1.5) {$b$};
					\node at (.6,-.35) {$x$};
				\end{tikzpicture}
			\end{center}
			By the triangle inequality for Lebesgue $p$-norms,
			\begin{eqnarray*}
				\Delta_p(A,B)&=&
				\left((y+z)^p+(v+w)^p\right)^{1/p}\\
				&=&\|(y,v)+(z,w)\|_p\\
				&\le& \|(y,v)\|_p+\|(z,w)\|_p\\
				&=&
				\left(y^p+v^p\right)^{1/p}
				+
				\left(z^p+w^p\right)^{1/p}\\
				&\le& ((y+b)^p+((v+a)^p)^{1/p}
				+ ((z+a)^p+(w+b)^p)^{1/p}\\
				&=&\Delta_p(A,C)+\Delta_p(C,B).\qedhere
			\end{eqnarray*}
		\end{proof}

		\begin{thm}
			 For each $1\le p\le\infty$, $\mathcal V_p$ is a metric.
		\end{thm}
		\begin{proof}
			By \Cref{fun-factory} and \Cref{mar29-22-11am},
			we only have to check $\abs{B\setminus A}\le \Delta_p(A,B)$, which is immediate for $1\le p\le\infty$.
		\end{proof}

		Of special interest may be $\mathcal V_2$ as a canonical interpolant between
		$\mathcal V_1$, the Jaccard distance, and
		$\mathcal V_{\infty}=J_{\infty}$, the analogue of the NID.
		If $\abs{B\setminus A}=3$, $\abs{A\setminus B}=4$, and $\abs{A\cap B}=5$, then
		\begin{eqnarray*}
			\mathcal V_1(A,B)&=&7/12,\\
			\mathcal V_2(A,B)&=&1/2,\\
			\mathcal V_{\infty}(A,B)&=&4/9.
		\end{eqnarray*}

		Note that if $A\subseteq B$ then $\mathcal V_p(A,B)=\mathcal V_1(A,B)$ for all $p$.

	\section{Conclusion and applications}

		Many researchers have considered metrics based on sums or maxima, but we have shown that these need not be considered in ``isolation''
		in the sense that they form the endpoints of a family of metrics.

		As an example, the mutations of spike glycoproteins of coronaviruses are of interest in connection with diseases such as CoViD-19.
		We calculated several distance measures between peptide sequences for such proteins.
		The distance
		\[
			Z_{2,\alpha}(x_0,x_1)=\alpha\min(\abs{A_1},\abs{A_2})+\clpha\max(\abs{A_1},\abs{A_2})
		\]
		where $A_i$ is the set of subwords of length 2 in $x_i$ but not in $x_{1-i}$,
		counts how many subwords of length 2 appear in one sequence and not the other.
%		Our calculations are available in a URL format as follows:\\

		%\noindent \verb!http://counter-automata.appspot.com/spike?metric=z2&alpha=0.36!\\

		\noindent We used the Ward linkage criterion for producing Newick trees using the \texttt{hclust} package for the Go programming language.
		The calculated phylogenetic trees were based on the metric $Z_{2,\alpha}$.

		We found one tree isomorphism class each for $0\le\alpha\le 0.21$, $0.22\le\alpha\le 0.36$, and $0.37\le\alpha\le 0.5$, respectively (\Cref{21}, \Cref{half}).
		\begin{figure}
			\begin{subfigure}[b]{0.4\textwidth}
				\includegraphics[width=5cm]{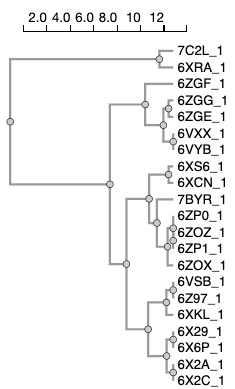}
			\end{subfigure}
			\begin{subfigure}[b]{0.4\textwidth}
				\includegraphics[width=5cm]{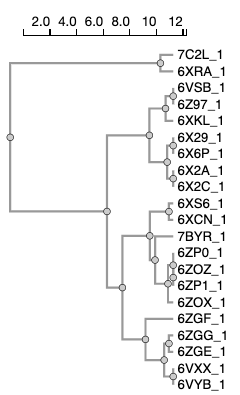}
			\end{subfigure}
			\caption{$\alpha=0.21$ and $0.36$.}\label{21}
		\end{figure}
		In \Cref{half} we are also including the tree produced using the Levenshtein edit distance in place of $Z_{2,\alpha}$.
		We see that the various intervals for $\alpha$ can correspond to ``better'' or ``worse'' agreement with other distance measures.
		Thus, we propose that rather than focusing on $\alpha=0$ and $\alpha=1/2$ exclusively, future work may consider the whole interval $[0,1/2]$.
		\begin{figure}
			\begin{subfigure}[b]{0.4\textwidth}
			\includegraphics[width=4cm]{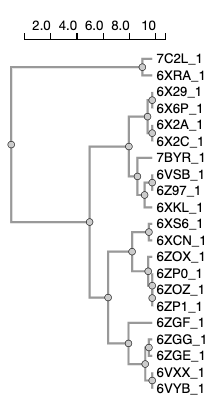}
			\end{subfigure}
			\begin{subfigure}[b]{0.4\textwidth}
			\includegraphics[width=4cm]{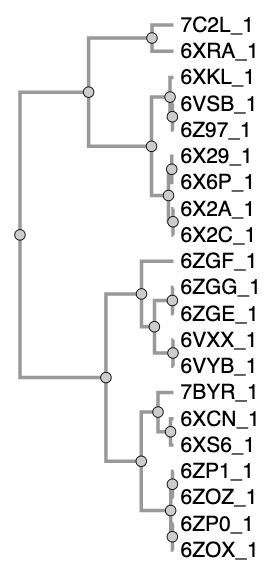}
			\end{subfigure}
			\caption{$\alpha=0.5$ and edit distance.}\label{half}
		\end{figure}

\bibliographystyle{plain}
\bibliography{nid}
\end{document}